\documentclass[10pt]{amsart}
\usepackage{color}
\usepackage[TS1,T1]{fontenc}
\usepackage{amsfonts,amssymb,amsmath,amsthm,enumerate,latexsym}

\usepackage{esint}
\usepackage{ulem}
\usepackage{epsfig}
\usepackage[mathscr]{eucal}
\usepackage{stackrel}
\usepackage{hyperref}

\usepackage{mathptmx} 
\usepackage[scaled=0.90]{helvet} 
\usepackage{courier} 
\normalfont
\usepackage[T1]{fontenc}

\def\R{{\mathbb {R}}}

\newcommand{\p}{\partial}
\newcommand{\defeq}{\mathrel{\mathop:}=}

\newtheorem{theorem}{Theorem}[section]

\newtheorem{lemma}[theorem]{Lemma}

\newtheorem{corollary}[theorem]{Corollary}

\theoremstyle{definition}

\newtheorem{definition}[theorem]{Definition}

\newtheorem{example}[theorem]{Example}

\theoremstyle{remark}

\newtheorem{remark}[theorem]{Remark}

\def \dist {\mathrm{dist}}

\numberwithin{equation}{section}

\begin{document}

\title[Parabolic problems with strong absorption]
{Regularity for degenerate evolution equations with strong absorption}

\author[Jo\~{a}o Vitor da Silva, Pablo Ochoa and Analia Silva]{Jo\~{a}o Vitor da Silva, \quad Pablo Ochoa  \\ $\&$ \\ Analia Silva}
\address[Jo\~{a}o Vitor da Silva]{Universidad de Buenos Aires. Departamento de Matem\'atica - FCEyN - Ciudad Universitaria, Pabell\'on I (1428). Buenos Aires, Argentina.}
\address[Pablo Ochoa]{Universidad Nacional de Cuyo and CONICET, Mendoza
5500, Argentina}
\address[Analia Silva]{Instituto de Matem\'{a}tica Aplicada San Luis-IMASL. Universidad Nacional de San Luis and CONICET. Ejercito de los Andes 950. D5700HHW San Luis. Argentina.}
\email[J.V. da Silva]{jdasilva@dm.uba.ar}
\email[P. Ochoa]{ochopablo@gmail.com}
\email[A. Silva]{asilva@dm.uba.ar}
\subjclass[2000]{35B53, 35B65, 35J60, 35K55, 35K65}
\keywords{$p$-Laplacian type operators, dead-core problems, sharp and improved intrinsic regularity, Liouville type results}

\begin{abstract}
 In this manuscript, we study geometric regularity estimates for degenerate parabolic equations of $p$-Laplacian type ($2 \leq p< \infty$) under a strong absorption condition:
$
   \Delta_p u - \frac{\partial u}{\partial t} =  \lambda_0 u_{+}^q  \quad \mbox{in} \quad \Omega_T \defeq \Omega \times (0, T),
$
where  $0 \leq q < 1$ and $\lambda_0$ is a function bounded away from zero and infinity. This model is interesting because it yields the formation of dead-core sets, i.e, regions where non-negative solutions vanish identically. We shall prove sharp and improved parabolic $C^{\alpha}$ regularity estimates  along the set $\mathfrak{F}_0(u, \Omega_T) = \partial \{u>0\} \cap \Omega_T$ (the free boundary), where $\alpha= \frac{p}{p-1-q}\geq 1+\frac{1}{p-1}$. Some weak geometric and measure theoretical properties as non-degeneracy, positive density, porosity and finite speed of propagation are proved. As an application, we prove a Liouville-type result for entire solutions provided their growth at infinity can be appropriately controlled. A specific analysis for Blow-up type solutions will be done as well.  The results obtained in this article via our approach  are new even for dead-core problems driven by the heat operator.
\end{abstract}

\maketitle


\section{Introduction}

\hspace{0.7cm}Throughout the last $40$ years a wide class of parabolic equations has been used  to model problems coming from chemical reactions, physical-mathematical phenomena, biological processes, population dynamics, among others. Some of the main topics approached are reaction-diffusion processes with one-phase transition. Thus, the existence of non-negative solutions plays an essential role in these studies (cf. \cite{BS}, \cite{Stak} and references therein). An enlightening prototype is the following model of an isothermal catalytic reaction-diffusion process:
\begin{equation}\label{DCP0}\left\{
\begin{array}{rclcl}
     \Delta u -\frac{\partial u}{\partial t}& = & f(u) & \mbox{in} & \Omega_T\\
     u(x, t) & = & g(x, t) & \mbox{on} & \partial \Omega \times (0, T)\\
     u(x, 0) & = & u_0(x) & \mbox{in} & \overline{\Omega},
\end{array}
\right.
\end{equation}
where the boundary data fulfils
$$
   0<u_0 \in C^0(\overline{\Omega}), \,\,\,g(x, t) = \mathfrak{j}>0\quad \text{and} \quad u(x, 0) = \mathfrak{j} \,\,\, \forall \,\, x\in \partial \Omega.
$$
Here, $u$ represents the concentration of a reactant evolving in time, where $\Omega_T \defeq \Omega \times (0, T)$, $\Omega \subset \R^N$ is a regular and bounded domain and $f$ is a non-linear convection term fulfilling $f(s) > 0$ if $s>0$ and $f(0) = 0$. Moreover, the boundary conditions mean that the reactant is injected with a constant isothermal flux on the boundary. Recall that in evolution problems  when $f$ is locally Lipschitz, it follows from the Maximum Principle that non-negative solutions must be strictly positive. However when $f$ is not Lipschitz (or  not decaying sufficiently fast) at the origin, then non-negative solutions may exhibit a plateau region, which is known in the literature as \textit{Dead-Core} set, i.e., a region of positive measure where non-negative solutions vanish identically.

 Advances in connection with existence theory of dead-core solutions, formation of dead-core regions and decay estimates at infinity were obtained in \cite{BS}, \cite{Guo}  and \cite{Sperb}.  Other properties of solutions such as growth of interfaces, shrink and estimates of/on the support and finite extinction in reaction-diffusion problems together with other qualitative properties may be found in D\'{i}az \textit{et al}'s fundamental articles \cite{AlvDiaz1}, \cite{AntDiazShma1}, \cite{DiazHer1}, \cite{DiazMing}, the Antontsev \textit{et al}'s classical book \cite{AntDiazShma2},  the survey \cite{Diaz2} and references therein. However, the lack of quantitative properties for $p-$parabolic dead-core problems constitutes our starting point in this research. In particular, we shall be interested in the derivation of quantitative results for the following class of parabolic dead-core problems of p-Laplacian type:

\begin{equation}\label{DCSA}
     \Delta_p \, u - \frac{\partial u}{\partial t} = \lambda_0(x,t).u_{+}^q(x, t)  \quad \mbox{in} \quad \Omega_T,
\end{equation}
with suitable boundary data, where $u_{+}  = \max\{u, 0\}$, $0\leq q<1$, $p \geq 2$, $\Omega \subset \R^N$ is a bounded smooth domain and  $\lambda_0$ (the \textit{Thiele modulus}) is bounded away from zero and infinity. We refer the reader to Section \ref{Prelim} for notation and definitions.

An important feature  in free boundary problems consists of Non-degeneracy of solutions (cf. \cite{LRS}, \cite{daSO}, \cite{SS}, \cite{Shah} and \cite{Tei4} for some examples in elliptic and parabolic settings). Our first main result gives the precise rate at which weak solutions $u$ to \eqref{DCSA} detach from their free boundaries.

\begin{theorem}[{\bf Non-degeneracy}]\label{NonDeg} Let $u$ be a weak solution to \eqref{DCSA}. Then for every $(x_0, t_0) \in  \overline{\{u>0\}}$ and $r>0$ such that $Q_r(x_0, t_0) \Subset \Omega_T$ there holds:
$$
    \displaystyle \sup_{\partial_p Q^{-}_{r}(x_0, t_0)} u(x, t) \geq \mathfrak{C}^{\ast}_0.r^{\frac{p}{p-q-1}},
$$
for a universal constant $\mathfrak{C}^{\ast}_0>0$.
\end{theorem}

Throughout this paper \textit{universal constants} are the ones depending only on dimension and structural properties of the problem, i. e. $N, p, q$ and bounds of $\lambda_0$.

 The proof of Theorem \ref{NonDeg} consists of combining the construction of an appropriate barrier function
$$
    \Phi(x, t)  = \mathfrak{a}.\Big( |x|^{\frac{p}{p-1}} + \mathfrak{b}t^{\frac{p-1-q}{(p-1)(1-q)}}\Big)^{\frac{p-1}{p-1-q}},
$$
for universal constants $\mathfrak{a}, \mathfrak{b}>0$, together with a version of comparison principle to dead-core solutions. Observe that the assumption $0 \leq q <1$ is necessary to guarantee that the barrier does not blow up as $t$ tends to $0$.

We also establish an upper growth estimate for dead-core solutions of \eqref{DCSA}.

\begin{theorem}[{\bf Improved regularity along free boundary}]\label{MThm0}
Let $u$ be a non-negative and bounded weak solution to \eqref{DCSA}, such that $\frac{\partial u}{\partial t} \geq 0$ a.e. in $\Omega_T$. Then, for every compact set $\mathrm{K} \Subset \Omega_T$ and for every $(x_0, t_0)\in \partial \{u > 0\} \cap \mathrm{K}$, there holds:
$$
     u(x, t) \leq \mathfrak{C}.\|u\|_{L^{\infty}(\Omega_T)}\left(|x-x_0|+|t-t_0|^{\frac{1}{\theta}}\right)^{\frac{p}{p-q-1}},
$$
for a constant $\mathfrak{C} >0$ depending only on universal parameters and $\dist(\mathrm{K}, \partial_p \Omega_T)$ and for all $(x, t)$ sufficiently close to $(x_0, t_0)$.
\end{theorem}

We briefly discuss some heuristic remarks related to the content of Theorem \ref{MThm0}. Firstly, the assumption of monotonicity in  time  is natural in the literature, see for example \cite{daSO}, \cite{Shah} and \cite{Shah07}, and it is not restrictive, because it can be deduced from the boundary conditions, for example.

Recall that the lack of Strong Minimum Principle, namely when $q<1$, might lead  to the existence of plateau regions to solutions of \eqref{DCSA}. In effect, the following profiles
\begin{eqnarray*}
 u(x, t) \defeq [(1-q)\lambda_0.(t_0-t)]_{+}^{\frac{1}{1-q}} \quad \mbox{or} \quad u(x, t) \defeq  \left[\lambda_0. \frac{(p-q-1)^{p}}{p^{p-1}(pq+N(p-1-q))}\right]^{\frac{1}{p-q-1}}.(\pm x)_{+}^{\frac{p}{p-q-1}}
\end{eqnarray*}
are weak dead-core solutions to \eqref{DCSA} with $\lambda_0(x, t) = \lambda_0$. This shows that that dead-core solutions may exist.  With respect to regularity, let us recall that in general (bounded) solutions to \eqref{DCSA} are parabolically $C^{1+ \alpha}$ for some $0 <\alpha\leq 1$ (cf. \cite{DB93}, \cite{DiBUV}, and \cite{Urbano}). Furthermore, a local $C^{\alpha^{\sharp}}$ behaviour for $\alpha^{\sharp} = 1+ \frac{1}{p-1}$ is the better regularity estimates expected for weak solutions with bounded right hand side (compare with \cite{ALS} in the context of elliptic and parabolic $p-$obstacle problems and \cite{ATU1, ATU2} for the elliptic context of the so termed $C^{p^{\prime}}$ regularity conjecture).

Notice that the dead-core analysis carries on an insightful and deep significance: solutions to \eqref{DCSA} have an improved pointwise behaviour as following:
$$
   \displaystyle \sup_{Q_r(x_0, t_0)} u(x, t) \approx r^{\alpha_0(p, q)} \quad \text{for} \quad \alpha_0(p, q) \defeq \frac{p}{p-q-1},
$$
around  free boundary points. Therefore, Theorem \ref{MThm0} yields that solutions are expected to be $C_{x, t}^{\left\lfloor \alpha_0\right\rfloor, \frac{\left\lfloor \alpha_0\right\rfloor}{\theta}}$ at free boundary points, where $\theta$ is the \textit{intrinsic scaling factor} in time variable:
$$
   \theta = \theta(p, \alpha_0) \defeq p + \alpha_0(2-p).
$$
We point out that for $p \geq 2$:
$$
   \theta(p, q) = \frac{p(1-q)}{p-q-1}>0 \quad \Leftrightarrow \quad 0 \leq q<1,
$$
which assures once again that the assumption $0 \leq q < 1$ must be present. Moreover, notice that:
\begin{eqnarray*}
 \alpha_0 = \frac{p}{p-q-1} > 1+\frac{1}{p-1}=\alpha^{\sharp} \quad \Leftrightarrow \quad q > 0.
\end{eqnarray*}Thus, we derive better regularity estimates along $\partial\{u>0\} \cap \Omega_T$.  Also, observe that $\alpha_0>2$ provided $2\leq p<4$, i.e., weak solutions become classical solutions. In summary, thanks to Theorem \ref{MThm0} we can access better regularity estimates on free boundary points than those available currently  (cf. \cite{ATU1, ATU2}, \cite{DB93}, \cite{DiBUV} and \cite{Urbano}). It is worth highlighting that such an approach has become a mainstream line of investigation in the  literature. We must mention, as one of the starting points of this theory, the Teixeira's work \cite{Tei1} where improved regularity estimates are obtained in the context of degenerate elliptic PDEs precisely along their set of critical points.

The technique applied in the proof of Theorem \ref{MThm0} can be adjusted in order to yield other relevant consequences. In this direction, we are able to establish a sharp gradient decay at free boundary points for a particular class of dead-core solutions, see Lemma \ref{IRresult2} for more details.

Our article has been strongly influenced by the \textit{geometric tangential analysis}, \textit{intrinsic scaling technique} and the theory of \textit{geometric free boundary problems} (a list of surveys related to these topics is \cite{CafSal},  \cite{CW}, \cite{SP}, \cite{ST}, \cite{PetrShah}, \cite{PT}, \cite{Shah07}, \cite{Tei2}, \cite{Tei5}, \cite{TU}, and \cite{Urbano}). Particularly, we should quote \cite{KKPS}, \cite{Shah} and \cite{Tei4}, which are ground-breaking manuscripts in the modern theory of free boundary and dead-core problems. Some of the core ideas further developed in this current article were introduced in such pivotal references.

We point out some final comments. The main obstacle in obtaining the estimates from Theorem \ref{MThm0}  for the singular case $1<p<2$ is that the technique applied to prove the Non-degeneracy property, does not work  for  the singular setting. Finally, we highlight that the methods employed through this paper only explore the structure and intrinsic scaling of the $p$-Laplacian operator. Indeed, the only requirement is the compactness argument, \textit{a priori} estimates and the strong maximum principle for the limiting equation. Consequently, the proofs can be adapted to more general parabolic problems of $p$-Laplacian type.

Our article is organized as following: In  Section \ref{Prelim} we shall present some preliminaries tools. An appropriated notion of weak solution is enunciated. We will also refer some results about comparison and we introduce a family of solutions invariant under intrinsic scaling. In Section \ref{SIRE}, we will prove a central result in our article, Lemma \ref{LemmaIter}, which allows us to place solutions in a flatness improvement regime.  At the end of this section, we prove Theorem \ref{MThm0}. Similar growth rate for gradient along free boundary points will also be performed, Lemma \ref{IRresult2}. Section \ref{NondFP} will be devoted to prove Non-degeneracy property, Theorem \ref{NonDeg}, and some of its consequences. Next, Section \ref{FurtAppCons} will be dedicated to present some applications of the main results. Blow-up analysis and some Liouville type results are proved.

\section{Preliminaries}\label{Prelim}

In this section we shall present the main tools and definitions that will be used through the article.

For $x_0 \in \mathbb{R}^{N}$ and $r >0$, we denote by $B_r(x_0)$ the Euclidean open ball with center $x_0$ and radius $r$. Also, for a point $(x_0, t_0) \in \Omega \times \R$ and $r > 0$, we consider three kinds of parabolic cylinders
$$
\begin{array}{rcll}
  Q_r(x_0, t_0) & \defeq  & B_r(x_0) \times (t_0-r^{\theta}, t_0+r^{\theta}) &  (\text{whole cylinder})\\
  Q^{+}_r(x_0, t_0) & \defeq & B_r(x_0) \times [t_0, t_0+r^{\theta}) & (\text{the upper semi-cylinder}) \\
  Q^{-}_r(x_0, t_0) & \defeq & B_r(x_0) \times (t_0-r^{\theta}, t_0] & (\text{the lower semi-cylinder}),
\end{array}
$$
where $\theta$ is the \textit{intrinsic scaling factor} in time variable:
$$
   \theta = \theta(p, q) \defeq \frac{p(1-q)}{p-1-q}.
$$
Moreover, we will omit the center of the cylinder as $(x_0, t_0) = (0, 0)$.

For a parabolic domain $ Q \defeq \Omega \times \mathcal{I}$, where $\mathcal{I}$ is an interval with endpoints $\mathfrak{a} < \mathfrak{b}$, we define the parabolic boundary by: $\partial_p Q \defeq (\overline{\Omega} \times \{\mathfrak{a}\})\cup (\partial \Omega \times \mathcal{I}).$

For $p>2$ and every $\mathrm{K} \Subset \Omega_T$ we have according to \cite[Chapter III]{DB93} the following intrinsic parabolic distance:
$$
\displaystyle \dist(\mathrm{K}, \partial_p \Omega_T) \defeq \inf_{(x, t) \in \mathrm{K} \atop{(y, s)\in \partial_p \Omega_T}} \left[|x-y|+\|u\|^{\frac{p-2}{p}}_{L^{\infty}\left(\Omega_T\right)}|t-s|^{\frac{1}{p}}\right]
$$

In the following, we give a notion of weak solution, see \cite[page 11]{DB93}.

\begin{definition}[\textbf{Weak solution-Steklov average}] A function $u \in \mathrm{V}^{1, p}(\Omega_T) = L^{\infty}((0,T); L^{1}(\Omega)) \cap L^{p}_{loc}((0,T);W^{1,p}_{loc}(\Omega))$ is said to be a weak solution to \eqref{DCSA} if, for every compact set $K \subset \Omega$ and every $0 < t<T-h$ there holds:
$$
  \int_{K \times \{t\}}\left[\frac{\partial u_h}{\partial t}+(|\nabla u|^{p-2}\nabla u)_h \cdot \nabla \zeta \right] dx=\int_{K\times \{t\}}f_h(x, u)\zeta dx
$$
for all $\zeta \in W^{1, p}_0(K),$ where $f_h(x, u) = (\lambda_0u_{+}^q)(x, t)$ and the \textit{Steklov average} of a function $v$ is defined as follows:
\begin{equation*}
  v_h(x, t) \defeq  \left\lbrace
  \begin{array}{rcl}
    \displaystyle \frac{1}{h}\int_t^{t+h}v(x, \mu)d\mu &  \text{ if } & t\in (0, T-h]\\
   0   & \text{ if } & t \in (T-h, T).
  \end{array}
  \right.
\end{equation*}
We point out that one takes the Steklov average since $\frac{\partial u}{\partial t}$ may not exist as a function in $\mathrm{V}^{1, p}$.
\end{definition}

We also quote a useful comparison result for weak solutions.

\begin{lemma}[{\bf Comparison Principle \cite[Theorem 1]{Diaz2}}]\label{comparison} Let $f \in C([0, \infty))$ be a non-negative and non-decreasing function. Assume that we have weakly:
$$
  \Delta_p \, u -\frac{\partial u}{\partial t} - f(u) \leq 0\leq \Delta_p \, v - \frac{\partial v}{\partial t} - f(v) \quad \mbox{in} \quad \Omega_T.
$$
If $v \leq u $ in $\partial_p \Omega_T$ then $v \leq u$ in $\Omega_T$.

\end{lemma}

Hereafter we shall adopt the following notation:
$$
   \displaystyle \mathcal{S}_{(r, x_0, t_0)}[u] \defeq  \sup_{Q^{-}_r(x_0, t_0)} u(x, t).
$$
Moreover, for notational simplicity, we will omit the $(x_0, t_0)$ when the center of the cylinder is the origin.

Throughout the article we shall consider the following family of solutions which are invariant under intrinsic scaling.

\begin{definition}\label{Class functions} Let $Q_1$ be the unit cylinder in $\mathbb{R}^{N+1}$. For any $2\leq p< \infty$ we say that $u \in \mathfrak{J}_p(\lambda_0, q)(Q_1)$ if:
\begin{enumerate}
  \item[\checkmark] $\Delta_p u-\frac{\partial u}{\partial t}  = \lambda_0(x, t)u^q_{+}(x, t) \quad \mbox{in} \quad Q_1$, for $0\leq q< \min\{1, p-1\} = 1$.
  \item[\checkmark] $0 \leq u \leq 1, \,\,\,0< \mathfrak{m} \leq \lambda_0 \leq \mathfrak{M} \quad \mbox{in} \quad Q_1$.
  \item [\checkmark] $\frac{\partial u}{\partial t}\geq 0$ a.e. in $Q_1$.
  \item[\checkmark] $u(0, 0)=0$.
\end{enumerate}
\end{definition}

In the next, we shall also consider  for $u \in \mathfrak{J}_p(\lambda_0, q)(Q_1)$ the following set:

$$
  \mathbb{V}_{p, q}[u] \defeq \left\{j \in \mathbb{N}\cup\{0\}; \mathcal{S}_{\frac{1}{2^j}}[u] \leq 2^{\frac{p}{p-1-q}}\max\left\{1,\frac{1}{\mathfrak{C}^{\ast}_0} \right\}\mathcal{S}_{\frac{1}{2^{j+1}}}[u] \right\},
$$
where $\mathfrak{C}^{\ast}_0 > 0$ is the constant in Theorem \ref{NonDeg}. Moreover, observe that $\mathbb{V}_{p,q}[u]$ is not empty. Indeed, $j=0\in \mathbb{V}_{p,q}[u]$ since, in view of Theorem \ref{NonDeg}:

$$
   \mathcal{S}_{\frac{1}{2}}[u] \geq \mathfrak{C}^{\ast}_0\left(\frac{1}{2}\right)^{\frac{p}{p-1-q}} \geq \mathfrak{C}^{\ast}_0\left(\frac{1}{2}\right)^{\frac{p}{p-1-q}}  \mathcal{S}_{1}[u],
$$
which shows that:
$$
  \mathcal{S}_{1}[u]\leq 2^{\frac{p}{p-1-q}}\max\left\{1,\frac{1}{\mathfrak{C}^{\ast}_0} \right\}\mathcal{S}_{\frac{1}{2}}[u].
$$

\section{Sharp improved regularity estimates}\label{SIRE}

\hspace{0.7cm}Through this Section we will prove an improved regularity result for weak solutions to \eqref{DCSA} along their free boundaries. We begin by estimating the growth rate of functions in $\mathfrak{J}_p(\lambda_0, q)(Q_1)$ near the free boundary.

\begin{lemma}\label{LemmaIter} There exist a positive constant $\mathfrak{C}_0 =  \mathfrak{C}_0(N, p, q, \mathfrak{M})$ such that:

  \begin{equation}\label{Eqiter}
     \mathcal{S}_{\frac{1}{2^{j+1}}}[u] \leq \mathfrak{C}_0.\left(\frac{1}{2^j}\right)^{\frac{p}{p-1-q}}
  \end{equation}
  for all $u \in \mathfrak{J}_p(\lambda_0, q)(Q_1)$ and $j \in \mathbb{V}_{p, q}[u]$.
\end{lemma}

\begin{proof} The proof follows the lines of \cite[Lemma 2.3]{Shah} adapted to our framework, and it will be achieved  by contradiction. Then, let us suppose that the thesis of the lemma fails. This means that for each $k \in \mathbb{N}$ we may find $u_k \in \mathfrak{J}_p(\lambda_0, q)(Q_1)$ and $j_k \in \mathbb{V}_{p, q}[u_k]$ such that:
  \begin{equation}\label{Eqcont}
     \mathcal{S}_{\frac{1}{2^{j_k+1}}}[u_k] > k \left(\frac{1}{2^{j_k}}\right)^{\frac{p}{p-1-q}}.
  \end{equation}
Now, for $\mathfrak{a}_k \defeq \frac{1}{2^{j_kp}}\frac{1}{\mathcal{S}^{p-2}_{\frac{1}{2^{j_k+1}}}[u_k]}$ we define the auxiliary function:
$$
   v_k(x, t) \defeq \frac{u_k\left(\frac{1}{2^{j_k}}x, \mathfrak{a}_k.t\right)}{\mathcal{S}_{\frac{1}{2^{j_k+1}}}[u_k]} \quad \mbox{in} \quad Q^{-}_1.
$$
Thus, $v_k$ fulfils:

\begin{enumerate}
  \item[\checkmark]  $v_k(0, 0) = 0$ and
  $$
    0 \leq v_k(x, t) \leq \frac{\mathcal{S}_{\frac{1}{2^{j_k}}}[u_k]}{\mathcal{S}_{\frac{1}{2^{j_k+1}}}[u_k]} \leq \mathfrak{A} \defeq 2^{\frac{p}{p-1-q}}\max\left\{1,\frac{1}{\mathfrak{C}^{\ast}_0} \right\} \quad \mbox{in} \quad Q^{-}_1.
$$
  \item[\checkmark] $\mathcal{S}_{\frac{1}{2}}[v_k] \geq 1$ $\left(\text{a consequence of}\,\,\, -(2^{-1})^{\theta}\mathfrak{a}_k \geq -(2^{-j_k-1}), \,\,\,u_k \geq 0 \,\,\,\mbox{and}\,\,\, \frac{\partial u_k}{\partial t} \geq 0\right)$.
  \item[\checkmark] $\frac{\partial v_k}{\partial t}\geq 0$ a.e. in $Q^{-}_1$.
  \item[\checkmark] $ \Delta_p v_k-\frac{\partial v_k}{\partial t} = \frac{1}{2^{j_k.p}}\frac{1}{\mathcal{S}^{p-1-q}_{\frac{1}{2^{j_k+1}}}[u_k]}\lambda_0\left(\frac{1}{2^{j_k}}x, \mathfrak{a}_k.t\right)(v_k)^q_{+}(x, t) \quad \mbox{in} \quad Q^{-}_1$.
\end{enumerate}
Therefore:
$$
   \left\|\Delta_p v_k-\frac{\partial v_k}{\partial t}\right\|_{L^{\infty}(Q^{-}_1)}  = \left\|\frac{1}{2^{j_k.p}}\frac{1}{\mathcal{S}^{p-1-q}_{\frac{1}{2^{j_k+1}}}[u_k]}\lambda_0(\cdot, \cdot)(v_k)^q_{+}\right\|_{L^{\infty}(Q^{-}_1)} \leq \mathfrak{A}^{q}.\sup_{Q_1} \lambda_0(x, t)\left(\frac{1}{k}\right)^{p-1-q} \to 0 \quad \mbox{as} \quad k \to \infty.
$$
The previous sentences together with standard compactness arguments for parabolic $p$-Laplacian operators (cf.  \cite[Lemma 14.1- p.75 and Lemma 14-(iii)-p.115]{DB93}) imply that, up to a subsequence, $v_k
\rightarrow v$ locally uniformly in
$\overline{Q^{-}_{\frac{4}{5}}}$. Furthermore, the limit fulfils:
\begin{enumerate}
  \item[\checkmark] $\Delta_p v - \frac{\partial v}{\partial t} = 0 \quad \mbox{in} \quad Q^{-}_{\frac{4}{5}}$.
  \item[\checkmark] $0\leq v \leq \mathfrak{A} \,\,\, \mbox{and} \,\,\,\frac{\partial v}{\partial t} \geq 0 \quad \text{a.e.}\quad \mbox{in} \quad Q^{-}_{\frac{4}{5}}$.
  \item[\checkmark] $\displaystyle \mathcal{S}_{\frac{1}{2}}[v] \geq 1.$
  \item[\checkmark] $v(0, t) = 0 \,\,\, \forall \,\,t \in \left(-\left(\frac{4}{5}\right)^{\theta}, 0\right]$.
\end{enumerate}
From now on we shall analyse two possibilities:

\begin{enumerate}
 \item({\large \bf Case $p=2$})

Notice that $u$ is a non-zero, non-negative caloric function in $Q^{-}_{\frac{4}{5}}$, non-decreasing in time and taking a local minimum at the origin. Hence, the strong minimum principle assures that $u$ is identically zero, which clearly yields a contradiction.

\item({\large \bf Case $p > 2$})

Due to the absence of the strong minimum principle, it will be necessary to analyse the previous limiting problem in a different manner.

Indeed, since $\frac{\partial v}{\partial t} \geq 0$, the supremum of $v$ on the $0$-level (in time) must be bigger than or equal to the ones in the corresponding negative levels. Consequently:
\begin{equation}\label{eqlevel0}
  \displaystyle \sup_{B_{\frac{1}{2}}} v(x, 0) \geq 1.
\end{equation}

Now, we claim that $v$ is time independent, in other words, $\frac{\partial v}{\partial t} \equiv 0 \quad \mbox{in} \quad Q^{-}_{\frac{4}{5}}$. Indeed, if such a claim holds, then $u$ would be a non-negative $p$-harmonic function vanishing at the origin. The strong minimum principle (see \cite{Vaz}) says that $v \equiv0$. However, this contradicts \eqref{eqlevel0}.

Let us prove the claim. Select $(x, t), (x, s) \in Q^{-}_{\frac{1}{2}}$ and by using the property of $\mathbb{V}_{p, q}[u_k]$ and the $\beta$-H\"{o}lder regularity estimate (cf. \cite[Theorem 1-p.41-42]{DB93}):
$$
\begin{array}{ccl}
  |v_k(x, t)-v_k(x, s)| & = & \frac{\left|u_k\left(\frac{1}{2^{j_k}}x, \mathfrak{a}_kt\right)-u_k\left(\frac{1}{2^{j_k}}x, \mathfrak{a}_ks\right)\right|}{\mathcal{S}_{\frac{1}{2^{j_k+1}}}[u_k]} \\
   & \leq & \mathfrak{A}\frac{\left|u_k\left(\frac{1}{2^{j_k}}x, \mathfrak{a}_kt\right)-u_k\left(\frac{1}{2^{j_k}}x, \mathfrak{a}_ks\right)\right|}{\mathcal{S}_{\frac{1}{2^{j_k}}}[u_k]} \\
   & \leq & \frac{\gamma\mathfrak{A}\|u_k\|_{L^{\infty}\left(Q^{-}_{{2^{-j_k}}}\right)}}{\mathcal{S}_{\frac{1}{2^{j_k}}}[u_k]} \left[\frac{\mathfrak{a}_k\|u_k\|^{p-2}_{L^{\infty}\left(Q^{-}_{{2^{-j_k}}}\right)}|t-s|}{\dist\left(Q^{-}_{{2^{-(j_k+1)}}}, \partial_p Q^{-}_{{2^{-j_k}}}\right)^p}\right]^{\frac{\beta}{p}},
\end{array}
$$
where $\gamma>1$ does not depend on $\|u_k\|_{L^{\infty}\left(Q^{-}_{{2^{-j_k}}}\right)}$ and the above intrinsic parabolic distance between sets satisfies:
\begin{eqnarray*}
\begin{array}{ccl}
   \dist\left(Q^{-}_{{2^{-(j_k+1)}}}, \partial_p Q^{-}_{{2^{-j_k}}}\right) & := & \displaystyle \inf_{(x, t) \in Q^{-}_{{2^{-(j_k+1)}}}, \atop{ (y, s)\in \partial_p Q^{-}_{{2^{-j_k}}}}} \left[|x-y|+\|u_k\|^{\frac{p-2}{p}}_{L^{\infty}\left(Q^{-}_{{2^{-j_k}}}\right)}|t-s|^{\frac{1}{p}}\right] \\
   & \geq & \left(\frac{1}{2^{j_k+1}}\right)^{\frac{\theta}{p}}\|u_k\|_{L^{\infty}\left(Q^{-}_{{2^{-j_k}}}\right)}^{\frac{p-2}{p}}
\end{array}
\end{eqnarray*}
Finally:
$$
\begin{array}{ccl}
  |v_k(x, t)-v_k(x, s)| & \leq  & \gamma\mathfrak{A}\left[|t-s|\mathfrak{a}_k2^{(j_k+1)\theta}\right]^{\frac{\beta}{p}} \\
   & = &  \gamma\mathfrak{A}\left[|t-s|\frac{1}{2^{j_kp}}\mathcal{S}^{2-p}_{\frac{1}{2^{j_k+1}}}[u_k]2^{(j_k+1)\theta}\right]^{\frac{\beta}{p}} \\
   & \leq & \gamma\mathfrak{A}\left[|t-s|\frac{1}{2^{j_kp}}k^{2-p}\left(\frac{1}{2^{j_k}}\right)^{\frac{p(2-p)}{p-q-1}}
   2^{(j_k+1)\theta}\right]^{\frac{\beta}{p}}\\
   & = & \gamma\mathfrak{A}\left[2^{\theta}.|t-s|k^{2-p}\right]^{\frac{\beta}{p}} \to 0 \quad \mbox{as} \quad k \to \infty.
\end{array}
$$
Therefore, $v$ is time-independent, and we finish the proof.
\end{enumerate}
\end{proof}

\begin{remark}\label{FlatHip} Notice that Lemma \ref{LemmaIter} assures that there exists a universal constant $0<\delta_0 \ll 1$ (small enough) such that if $u \in \mathfrak{J}_p(\lambda_0, q)(Q_1)$ with:
$$
   \left\|\Delta_p u - \frac{\partial u}{\partial t}\right\|_{L^{\infty}(Q^{-}_1)} = \left\|\lambda_0u_{+}^q\right\|_{L^{\infty}(Q^{-}_1)} \leq \delta_0,
$$
then:
$$
   \mathcal{S}_{\frac{1}{2^{j+1}}}[u] \leq \mathfrak{C}_0.\left(\frac{1}{2^j}\right)^{\frac{p}{p-1-q}}
$$
for all $j \in \mathbb{V}_{p, q}[u]$.
\end{remark}

We next state the growth property of the elements in the class $\mathfrak{J}_p(\lambda_0, q)(Q_1)$.  Even though the proof mainly follows the guidelines of the proof of \cite[Theorem 2.2]{Shah}, a different barrier adapted to \eqref{DCSA} should be used. We quote the full proof for the convenience of the reader.

\begin{theorem}\label{MThm} There exists a positive constant $\mathfrak{C} =  \mathfrak{C}(N, p, q, \mathfrak{M})$ such that for all $u \in \mathfrak{J}_p(\lambda_0, q)(Q_1)$ there holds:
$$
  u(x, t) \leq \mathfrak{C}.\mathfrak{d}(x, t)^{\frac{p}{p-1-q}} \quad \forall \,\, (x, t) \in Q_{\frac{1}{2}},
$$
where:
$$
\mathfrak{d}(x, t) \defeq \left\{
\begin{array}{ccl}
  \displaystyle \sup\{r\geq 0; Q_r(x, t) \subset \{u>0\}\} & \mbox{for} & (x, t)\in \{u>0\} \\
  0 & \mbox{otherwise}. &
\end{array}
\right.
$$
\end{theorem}

\begin{proof}

 The proof will be achieved  by an inductive reasoning. First of all, we claim that:
\begin{equation}\label{EqFinalEst}
  \mathcal{S}_{\frac{1}{2^j}}[u] \leq \mathfrak{C}_0\left(\frac{1}{2^{j-1}}\right)^{\frac{p}{p-1-q}} \quad \forall \,\, j \in \mathbb{N},
\end{equation}
where $\mathfrak{C}_0$ is the constant coming from Lemma \ref{LemmaIter}. Note that if $\mathfrak{C}_0 \geq 1$, which we can suppose without loss of generality, then \eqref{EqFinalEst} holds for $j=0$. Suppose now that \eqref{EqFinalEst} holds for some $j \in \mathbb{N}$. We will verify the $(j+1)^{\text{th}}$ step of induction. In fact, if $j \in \mathbb{V}_{p, q}[u]$ then the result holds directly by Lemma \ref{LemmaIter}. On the other hand, if \eqref{EqFinalEst} fails, then we obtain by using the inductive hypothesis, the following estimates:

$$
   \mathcal{S}_{\frac{1}{2^{j+1}}}[u] \leq \left(\frac{1}{2}\right)^{\frac{p}{p-1-q}}.\mathcal{S}_{\frac{1}{2^j}}[u] \leq  \mathfrak{C}_0.\left(\frac{1}{2}\right)^{\frac{p}{p-1-q}}\left(\frac{1}{2^{j-1}}\right)^{\frac{p}{p-1-q}} = \mathfrak{C}_1.\left(\frac{1}{2^{j}}\right)^{\frac{p}{p-1-q}}.
$$
Therefore, \eqref{EqFinalEst} holds for all $j \in \mathbb{N}$.

Now, in order to finish the proof for a continuous parameter $r \in (0, 1)$ let $j \in \mathbb{N}$ be the greatest integer such that $\frac{1}{2^{j+1}} \leq r < \frac{1}{2^j}$. Then:

$$
 \mathcal{S}_{r}[u] \leq \mathcal{S}_{\frac{1}{2^{j}}}[u] \leq \mathfrak{C}_0.\left(\frac{1}{2^{j-1}}\right)^{\frac{p}{p-1-q}} \leq \mathfrak{C}(N, p, q, \mathfrak{M}).r^{\frac{p}{p-1-q}}.
$$

Finally, in order to obtain an estimate for $u$ over the whole cylinder we will use a suitable barrier function from above. Let:

$$
  \Phi(x, t)= \mathfrak{c}_1.\left(\mathfrak{a}.|x|^{\frac{p}{p-1}}+\mathfrak{b}.t^{\frac{p-1-q}{(p-1)(1-q)}}\right)^{\frac{p-1}{p-1-q}},
$$
where the first constant:
$$
  \displaystyle \mathfrak{c}_1 \defeq  \left(\frac{\mathfrak{m}}{\Big( N+\frac{pq}{p-1-q}\Big)} \Big(\frac{p-1-q}{\mathfrak{a}.p}\Big)^{p-1}\right)^{\frac{1}{p-1-q}}
$$
is chosen so that:
$$
  \Delta_p \Phi -\frac{\partial \Phi}{\partial t} - \lambda_0(x, t).\Phi^{q}\leq 0 = \Delta_p u- \frac{\partial u}{\partial t}- \lambda_0(x, t)u^{q}_{+}
$$
in $Q_1^{+}$, for all $\mathfrak{a}, \mathfrak{b}>0$. Finally, pick $\mathfrak{b}>\mathfrak{a}$ with $\mathfrak{a}$ large enough so that $\Phi \geq u$ on $\partial_p Q_1^{+},$ where we have used that $\mathcal{S}_r[u]\leq \mathfrak{c}.r^{\frac{p}{p-1-q}}$ for the estimate on $\{t = 0\}$. Consequently, the comparison principle (Lemma \ref{comparison}) implies that $\Phi \geq  u \,\,\, \mbox{in} \,\,\,Q_1^{+}$. Therefore:
\begin{equation}\label{EqFinEst}
  \displaystyle \sup_{Q_r} u(x, t) \leq  \mathfrak{C}(N, p, q, \mathfrak{m}).r^{\frac{p}{p-1-q}}.
\end{equation}
\end{proof}

We now are ready to prove the main result of the article.

\begin{proof}[\textbf{Proof of Theorem \ref{MThm0}}]\label{rem1} In order to prove  Theorem \ref{MThm0}, we have to reduce the hypothesis presented on it to the framework of Theorem \ref{MThm}. We assume without loss of generality that $\mathrm{K} = \overline{Q_1} \subset \Omega_T$. For $(x, t) \in \{u>0\} \cap \mathrm{K}$ let $\mathfrak{d}(x, t)$ the distance coming from  Theorem \ref{MThm}. For $(x_0, t_0) \in \partial \{u>0\} \cap \mathrm{K}$ let us define:
$$
     v(y, s)\defeq \frac{u\left( x_0 + \mathfrak{R}_0 y, t_0+\mathfrak{R}_0^{\theta}s\right)}{\kappa_0} \quad \mbox{in} \quad Q_1
$$
for $\kappa_0, \,\mathfrak{R}_0>0$ constants to be determined \textit{a posteriori}. From the equation satisfied by $u$, we easily verify that $v$ fulfils in the weak sense:
\begin{equation}\label{eqcomp}
  \Delta_p v - \frac{\partial v}{\partial t} = \hat{\lambda}_0(x, t).v^q_{+}(y, s),
\end{equation}for an appropriate bounded function $\hat{\lambda}_0(x, t)$. Now, let $\delta_{\lambda_0, p, q} = \delta_0>0$ be the greatest universal constant, granted by the Remark \ref{FlatHip} such the Lemma \ref{LemmaIter} holds provided:
$$
  \left\|\Delta_p v- \frac{\partial v}{\partial t}\right\|_{L^{\infty}(Q_1)} = \left\|\hat{\lambda_0} v_{+}^q\right\|_{L^{\infty}(Q_1)} \leq \delta_{\lambda_0, p, q}.
$$
Then, we make the following choices in the definition of $v$:
$$
   \kappa_0 \defeq \|u\|_{L^\infty(\Omega_T)} \quad \mbox{and} \quad \mathfrak{R}_0 := \min\left\{1, \frac{\text{dist}(\mathrm{K}, \partial_p \Omega_T)}{2}, \sqrt[p]{\kappa_0}, \sqrt[p]{\frac{\delta_0.\kappa_0^{p-1-q}}{\mathfrak{M}}}\right\}.
$$
Finally, with such selections $v$ fits into the framework of Theorem \ref{MThm}.
\end{proof}

 \vspace{0.4cm}

\begin{remark}In view of previous results, we must highlight the relationship between regularity coming from dead-core solutions and the one coming from the classical Schauder theory. For this end, let us suppose that $u$ is a classical solution to
\begin{equation}\label{eqDCvsST}
    \Delta u - \frac{\partial u}{\partial t}= \lambda_0(x, t)u^{q}_{+}(x, t) \quad \text{in} \quad Q_1,
\end{equation}
where $0<q<1$ and $\lambda_0 \in C^{0, q}(Q_1)$. Thus, under such assumptions the Schauder theory assures that solutions to \eqref{eqDCvsST} are $C_{\text{loc}}^{2 +q, \frac{2 +q}{2}}(Q_1)$ (particularly at free boundary points). On the other hand, our main Theorem  \ref{MThm} claims that $u$ is $C^{\kappa + \alpha , \frac{\kappa+\alpha}{2}}$ at free boundary points, where
$$
  \kappa \defeq \left\lfloor \frac{2}{1-q}\right\rfloor \quad \text{and} \quad \alpha :=\frac{2}{1-q} - \left\lfloor\frac{2}{1-q}\right\rfloor.
$$
Nevertheless, notice that for any $0<q<1$ one has
$$
  \frac{2}{1-q} > 2+q,
$$
which means that dead-core solutions are more regular, along free boundary points, that the best regularity result coming from classical regularity theory (compare with \cite{SP} for improved regularity estimates in the context of flat solutions and \cite[Theorem 1.1]{daSO} for similar estimates in the context of fully nonlinear parabolic dead core problems).
\end{remark}

\vspace{0.4cm}

Next, we shall analyse the regularity behaviour for weak solutions to the critical case $q=p-1$:
\begin{equation}\label{eqUCP}
  \Delta_p \, u - \frac{\partial u}{\partial t} = \lambda_0(x, t).u^{p-1}_{+}(x, t) \quad \mbox{in} \quad \Omega_T,
\end{equation}
where $0<\mathfrak{m} \leq \lambda_0 \leq \mathfrak{M}$ and $1<p< \infty$.

An interesting issue in this context is the following: Does the critical case permit the formation of dead-core regions? By means of an energy argument, we shall prove that a non-negative solution to \eqref{eqUCP} cannot vanish at interior points, unless it is identically zero.

\begin{theorem}[{\bf Strong Maximum Principle}]\label{strongmp} Let $u$ be a non-negative weak solution to \eqref{eqUCP} for $2\leq p< \infty$. If there exists a point $(x_0, t_0) \in \Omega_T$ such that $u(x_0, t_0)=0$, then $u \equiv 0$ in $\Omega_T$.
\end{theorem}

\begin{proof} From standard energy estimates we have:
$$
\frac{1}{2}\frac{d}{dt} \displaystyle \int_{\Omega} |u(x, t)|^2 dx+\int_{\Omega} |\nabla u|^p dx  = \int_{\Omega} \lambda_0(x, t)|u|^p dx \leq \mathfrak{M}\int_{\Omega} |u|^p dx.
$$
Moreover, by using the variational characterization of $\mathfrak{M}$, for example, via Poincar\'{e}'s inequality we obtain:
$$
   \frac{1}{2}\frac{d}{dt} \displaystyle \int_{\Omega} |u(x, t)|^2 dx \leq -\int_{\Omega} |\nabla u|^p dx + \mathfrak{M}\int_{\Omega} |u|^p dx \leq 0.
$$
Finally, we conclude that $u(x, t) = 0$ in $ \Omega_T$, since $u(x_0, t_0) = 0$.
\end{proof}

\begin{example} Theorem \ref{strongmp} assures that non-trivial weak solutions must be strictly positive when $q = p- 1$. As a matter of fact, fixed any direction $i = 1, \cdots, N$ and a constant $\lambda_0>0$. We then have that:

$$
u(x, t) \defeq \left\{
\begin{array}{lcc}
   e^{\sqrt[p]{\frac{\lambda_0}{p-1}}. x_i} & \mbox{for} & p\geq 2 \\
   \sqrt[2-p]{(p-2)\lambda_0.t} & \mbox{for} & p > 2\\
    e^{\sqrt{\lambda_0}x_i} + e^{-\lambda_0.t} & \mbox{for} & p=2.
\end{array}
\right.
$$
is a smooth and strictly positive solutions to:
$$
     \Delta_p u(x, t) - \frac{\partial u}{\partial t}(x, t) = \lambda_0.u^{p-1}(x, t) \quad \mbox{in} \quad \Omega_T.
$$
\end{example}

Finally, using Lemma \ref{LemmaIter} we are able to prove a similar growth rate for the gradient of a particular family of functions.

In order to prove the gradient estimate we must restrict ourselves to the following class of functions:

\begin{definition} We say that $v \in \mathfrak{R}^{\infty}_{p}(Q_1)$ provided:
\begin{enumerate}
  \item $v$ is a weak solution to:
  $$
    \Delta_p v- \frac{\partial v}{\partial t} = f \in L^{\infty}(Q_1)
  $$
  and fulfils the following \textit{a priori} estimate:
$$
  \mathcal{S}_{\frac{1}{2}}[|\nabla v(x,t)|] \leq \mathfrak{C}(N, p)\left[\|v\|_{L^{\infty}(Q_1)}+ \left\|f\right\|^{\frac{1}{p-1}}_{L^{\infty}(Q_1)}\right].
$$

  \item $\mathcal{S}_{1}[|\nabla v(x,t)|] \leq 1$.

  \item For each $j \in \mathbb{N}$ there exists a universal constant $\kappa_j>0$ with $\displaystyle \liminf_{k \to \infty} \kappa_j>0$ such that:
  $$
  \displaystyle \sup_{B_{\frac{1}{2^{j+1}}}\times \left(-\left(\frac{1}{2}\right)^{\theta}\alpha_j, 0\right)} |\nabla v(x, t)|\geq \kappa_j\mathcal{S}_{\frac{1}{2^{j+1}}}[|\nabla v(x,t)|] \quad \forall\,\, j \in \mathbb{N},
  $$
  where $\alpha_j \defeq \frac{1}{2^{2j}}.\frac{1}{\mathcal{S}^{p-2}_{\frac{1}{2^{j+1}}}[|\nabla v(x,t)|]}$.
\end{enumerate}
\end{definition}

Observe that for $p=2$ the class $\mathfrak{R}^{\infty}_{p}(Q_1) \neq \emptyset$. Moreover, in this case $\kappa_j = 1$ for all $j \in \mathbb{N}$ (compare with \cite[Section 2.3]{SP}, \cite[Corollary 4.1]{daSO}, \cite[Section 5 and 6]{ST} and the references therein).

\begin{lemma}\label{IRresult2} If $\mathfrak{R}^{\infty}_{p}(Q_1)$ is not empty. Then, there exists a positive constant $\mathfrak{C}_1 =  \mathfrak{C}_1(N, p, q, \mathfrak{M})$ such that for all $u \in \mathfrak{J}_p(\lambda_0, q)(Q_1) \cap \mathfrak{R}^{\infty}_{p}(Q_1)$ there holds:
$$
  |\nabla u(x, t)| \leq \mathfrak{C}_1.\mathfrak{d}(x, t)^{\frac{1+q}{p-1-q}} \quad \forall (x,t)\in \,\, Q_{\frac{1}{2}},
$$
\end{lemma}

\begin{proof} As before, it is enough to prove the following estimate:
\begin{equation}\label{eqEstGrad1}
  \mathcal{S}_{\frac{1}{2^{j+1}}}[|\nabla u(x,t)|] \leq \max\left\{\mathfrak{C}_2.\left(\frac{1}{2^j}\right)^{\frac{1+q}{p-1-q}}, \left(\frac{1}{2}\right)^{\frac{1+q}{p-1-q}}\mathcal{S}_{\frac{1}{2^{j}}}[|\nabla u(x,t)|]\right\},
\end{equation}
for all $j \in \mathbb{N}$ and a constant $\mathfrak{C}_2 = \mathfrak{C}_2(N, p, q, \mathfrak{M})$.

Let us suppose that \eqref{eqEstGrad1} does not hold. Then, there exists $u_j \in \mathfrak{J}_p(\lambda_0, q)(Q_1) \cap \mathfrak{R}^{\infty}_{p}(Q_1)$ such that:
\begin{equation}\label{eqEstGradHip}
  \mathcal{S}_{\frac{1}{2^{j+1}}}[|\nabla u_j(x,t)|] \geq \max\left\{j\left(\frac{1}{2^j}\right)^{\frac{1+q}{p-1-q}}, \left(\frac{1}{2}\right)^{\frac{1+q}{p-1-q}}\mathcal{S}_{\frac{1}{2^{j}}}[|\nabla u_j(x,t)|]\right\}.
\end{equation}
Next, we define the auxiliary normalized and scaled function:
$$
   v_j(x, t) \defeq \frac{2^ju_j\left(\frac{1}{2^j}x, \alpha_j t\right)}{\mathcal{S}_{\frac{1}{2^{j+1}}}[|\nabla u_j(x,t)|]}.
$$
Notice that due to \eqref{eqEstGradHip} we get:
$$
   \alpha_j <\left(\frac{1}{j}\right)^{p-1-q}.\frac{1}{2^{j(1-q)}} \to 0 \quad \text{as} \quad j \to \infty.
$$

Next, using \eqref{EqFinEst} and \eqref{eqEstGradHip} we obtain:
$$
  0\leq  v_j(x, t) \leq \frac{2^j\mathfrak{C}(2^{-j})^{\frac{p}{p-1-q}}}{\mathcal{S}_{\frac{1}{2^{j+1}}}[|\nabla u_j(x,t)|]}\leq \frac{\mathfrak{C}}{j} \quad \mbox{for} \quad (x, t) \in Q_1.
$$
Furthermore,
\begin{enumerate}
  \item[\checkmark] $ \Delta_p v_j - \frac{\partial v_j}{\partial t} = \hat{\lambda_0}(x, t)(v_j)^{q}_{+}(x, t) \quad \mbox{in} \quad Q_1$
in the weak sense, where
$$
  \hat{\lambda_0}(x, t) \defeq \frac{1}{2^{j(1+q)}}\frac{1}{\mathcal{S}^{p-1-q}_{\frac{1}{2^{j+1}}}[|\nabla u_j(x,t)|]}\lambda_0\left(\frac{1}{2^{j}}x, \alpha_j.t\right).
$$

  \item[\checkmark] $\displaystyle \mathcal{S}_{\frac{1}{2}}[|\nabla v_j(x,t)|] \geq \kappa_j>0$ due to $v_j \in \mathfrak{R}^{\infty}_{p}(Q_1)$.

\end{enumerate}

Consequently, we have that:
$$
   \left\|\hat{\lambda_0}.(v_j)^{q}_{+}\right\|_{L^{\infty}(Q_1)} \leq \mathfrak{M}.\mathfrak{C}^{q}.\left(\frac{1}{j}\right)^{p-1}.
$$
Finally, by using the \textit{a priori} gradient estimate from $\mathfrak{R}^{\infty}_{p}(Q_1)$ we obtain:
\begin{eqnarray*}
 \kappa_j \leq \mathcal{S}_{\frac{1}{2}}[|\nabla v_j(x,t)|] \leq \mathfrak{C}(N, p)\left[\|v_j\|_{L^{\infty}(Q_1)}+ \left\|\hat{\lambda_0}(v_j)^{q}_{+}\right\|^{\frac{1}{p-1}}_{L^{\infty}(Q_1)}\right] \leq \mathfrak{C}^{\ast}.\frac{1}{j} \to 0 \quad \mbox{as} \quad j \to \infty,
\end{eqnarray*}
which is a contradiction. Therefore the proof is ended.
\end{proof}


\section{Non-degeneracy and its consequences}\label{NondFP}

\hspace{0.6cm}Throughout this section we will deliver the proof of Non-degeneracy property for weak solutions. Thereafter, we will present  important consequences of this fundamental property.

\begin{proof}[\textbf{Proof of Theorem \ref{NonDeg}}] Notice that by continuity of weak solutions it suffices to show the thesis for points in $\{u > 0\}$. Fix $(x_0, t_0) \in \{u > 0\}$. Let $r >0$ be so that $Q_r(x_0, t_0) \subset Q_1$, and consider the scaled functions:

$$
   u_r(x,t) =  \frac{u(x_0 + rx, t_0 + r^{\theta}t)}{r^{\frac{p}{p-1-q}}}, \quad (x, t)\in \overline{Q^{-}_1}.
$$
Observe that:
$$
   \Delta_p u_r -\frac{\partial u_r}{\partial t} - \lambda_0(x_0 + rx, t_0 + r^{\theta}t) u_r^{q}  = 0 \quad \mbox{in } Q_1^{-}.
$$
We introduce a comparison function defined in $\overline{Q^{-}_1}$:
$$
   \hat{\Phi}(x, t) \defeq  \mathfrak{c}.\Big( |x|^{\frac{p}{p-1}} + (-t)^{\frac{p-1-q}{(p-1)(1-q)}}\Big)^{\frac{p-1}{p-1-q}},
$$
where
\begin{equation}\label{choice c_1}
\mathfrak{c} = \min \Big\{\Big[\frac{\mathfrak{m}(1-q)}{2}\Big]^{\frac{1}{1-q}}, \Big[ \frac{\mathfrak{m}}{2}\Big(\frac{p-1-q}{p}\Big)^{p-1}\Big(N+\frac{pq}{p-1-q}\Big)\Big]^{\frac{1}{p-1-q}}\Big\}.
\end{equation}
It follows from \eqref{choice c_1} that:
$$
  \Delta_p \hat{\Phi} - \frac{\partial \hat{\Phi}}{\partial t}-\lambda_0(x, t).\hat{\Phi}^{q}\leq 0.
$$
Now, if $u_r \leq \hat{\Phi}$ on the  parabolic boundary of $Q^{-}_1$, then Lemma \ref{comparison} would imply $u_r \leq \hat{\Phi}$ in $Q^{-}_1$. But this contradicts $\hat{\Phi}(0, 0)=0 < u_r(0, 0)$. Hence, there should be a point $(x^{\prime}, t^{\prime})$ on $Q^{-}_1$ so that:
$$
  u_r(x^{\prime}, t^{\prime}) \geq \hat{\Phi}(x^{\prime}, t^{\prime}).
$$
Scaling back proves the result.
\end{proof}

Such a Non-degeneracy property and the growth rate for weak solutions to \eqref{DCSA} will drive us to  establish some measure-theoretical properties of the free boundary. We start by showing a property of positive density.

\begin{corollary}[{\textbf{Positive Lebesgue density of $\{u > 0\}$}}]\label{positive density f} Let $u$ be the weak solution to \eqref{DCSA}. Then, there exists a  positive constant $\varrho= \varrho(N, p, \lambda_0, \|u\|_{L^{\infty}(Q_1)})$ such that for all $(x_0,t_0)\in  \overline{\{u > 0\}}$ and $0 <r< 1$ so that $Q_r(x_0,t_0)\subset Q_{\frac{1}{2}}$, the inclusion:
\begin{equation}
    Q_{\varrho r}(x^{\prime},t^{\prime})\subset Q_r(x_0,t_0)\cap \{u> 0\},
\end{equation}
holds for some $(x^{\prime},t^{\prime}) \in Q^{-}_r(x_0,t_0)$.

\end{corollary}

\begin{proof} Let $(x_0,t_0) \in \overline{\{u > 0\}} \cap \overline{Q_{\frac{1}{2}}}$. For $r$ small enough, we have by Theorem \ref{NonDeg} that there exists $(x^{\prime},t^{\prime})\in Q^{-}_{\frac{r}{2}}(x_0,t_0)$ such that:
\begin{equation}\label{nd}
    u(x^{\prime},t^{\prime})\geq \mathfrak{C}^{\ast}_0.\left( \frac{r}{2}\right)^{\frac{p}{p-q-1}}.
\end{equation}
Suppose that for all $0 < \mathfrak{d} < 1$ small, there exists a point $(x,t) \in \partial \{u > 0\} \cap \overline{Q_{\frac{1}{2}}}$ satisfying:
\begin{equation}\label{nd1}
   (x^{\prime},t^{\prime})\in Q_{\mathfrak{d}}(x,t)\subset Q_r(x_0,t_0).
\end{equation}
Now, according to \eqref{nd}, \eqref{nd1} and Theorem \ref{MThm}, it follows:
\begin{equation}
    \mathfrak{C}^{\ast}_0.\left(\frac{r}{2}\right)^{\frac{p}{p-q-1}} \leq u\left(x^{\prime},t^{\prime}\right)\leq \sup_{Q_{\mathfrak{d}}(x,t)} u \leq \mathfrak{C},\mathfrak{d}^{\frac{p}{p-q-1}}.
\end{equation}
This clearly does not hold for $\mathfrak{d} < \iota.\frac{r}{2}$,
where:
$$
      \iota \defeq \left( \frac{\mathfrak{C}^{\ast}_0}{\mathfrak{C}}\right)^{\frac{p-q-1}{p}} < 1.
$$
Hence
$$
     Q_{\frac{\iota}{4}r}(x^{\prime},t^{\prime})\subset Q_r(x_0,t_0) \cap \{u > 0\}.
$$
This ends the proof of the theorem.

\end{proof}

\begin{remark} Notice that Corollary \ref{positive density f} assures that the free boundary cannot have Lebesgue points. Consequently, for any compact set $\mathrm{K} \subset Q_1$, we have:
$$
   \mathscr{L}^{N+1}(\partial \{u>0\} \cap \mathrm{K}) = 0.
$$
\end{remark}

Next, we shall prove, as an easy consequence of the above result, that the free boundary is a porous set. We recall the definition of this notion.

\begin{definition}[{\bf Porous set}]\label{porousset} A set $\mathcal{E} \in \mathbb{R}^{N}$ is said to be porous with porosity constant $0<\varsigma\leq 1$ if there exists $\mathfrak{R} > 0$ such that for each $x_0 \in \mathcal{E}$ and $0 < \mathfrak{r} < \mathfrak{R}$ there is a point $x^{\prime}$ so that $B_{\mathfrak{r}\varsigma}(x^{\prime}) \subset B_{\mathfrak{r}}(x_0) \setminus \mathcal{E}$.
\end{definition}

Observe that a porous set has Hausdorff dimension at most $N-c_0\varsigma^N$, where $c_0 = c_0(N)>0$. In particular, a porous set has Lebesgue measure zero (cf. \cite{Z88}).

\begin{corollary}[{\bf Porosity for $t$-level of free boundary}] Let $u$ be the weak solution to \eqref{DCSA}. For every compact set $\mathrm{K} \subset Q_1$ holds that:
$$
   \mathcal{H}^{N-\delta}(\partial \{u>0\} \cap \mathrm{K} \cap \{t=t_0\})< \infty
$$
for a constant $0< \delta = \delta(N, p, q, \lambda_0,
\|u\|_{L^{\infty}(Q_1)}, \dist(\mathrm{K}, \partial_p
Q_1))\leq 1.$
\end{corollary}

\begin{proof}
Without loss of generality we can suppose that $\mathrm{K} = \overline{Q_{\frac{1}{2}}}$. Let $(z, t_0)\in\p \{u>0\}\cap \overline{Q_{\frac{1}{2}}}$ then for $0<r\ll1$, according to Non-degeneracy property, there exists $x^{\prime}\in\p B_r(z)$ such that:
$$
   u(x^{\prime},t_0)\geq \mathfrak{C}^{\ast}_0.r^\frac{p}{p-1-q}.
$$
On the other hand, from Theorem \ref{MThm}:
$$
   u(x^{\prime}, t_0)\leq \mathfrak{C}.\mathfrak{d}(x^{\prime}, t_0)^\frac{p}{p-1-q}.
$$
Consequently:
$$
  \mathfrak{C}^{\ast}_0.r^\frac{p}{p-1-q} \leq u(x^{\prime},t_0)\leq \mathfrak{C}.\mathfrak{d}(x^{\prime}, t_0)^\frac{p}{p-1-q}.
$$
Next, by selecting $\delta^{\ast} =(\frac{\mathfrak{C}^{\ast}_0}{\mathfrak{\mathfrak{C}}})^\frac{p-1-q}{p}$, then $\mathfrak{d}(x^{\prime}, t_0)\geq \delta^{\ast}r$ for a $0<\delta^{\ast}\leq 1$. Therefore:
$$
   B_{\delta^{\ast}r}(x^{\prime})\cap B_r(z)\subset \{u>0\} \cap \Omega_T.
$$
Now, choose $y\in [z, x^{\prime}]$ such that $|y-x^{\prime}|=\frac{\delta^{\ast} r}{2}$. Note that for any $y_0\in B_{\frac{\delta^{\ast} r}{2}}(y)$ we have:
$$
  |y_0-x^{\prime}| \leq |y_0-y|+|y-x^{\prime}| = \delta^{\ast} r.
$$
Moreover, since $|z-x^{\prime}| = |y-z| + |y-x^{\prime}|$ then:
$$
   |y_0 - z| \leq |y_0 - y|+ (|z-x^{\prime}|-|y - x^{\prime}|)\leq\frac{\delta^{\ast} r}{2}+\left(r-\frac{\delta^{\ast} r}{2}\right) = r
$$
so, we conclude that:
$$
  B_{\frac{\delta^{\ast} r}{2}}(y)\subset B_{\delta^{\ast} r}(x^{\prime}) \cap B_r(z)\subset B_r(z)\setminus \partial \{u>0\}.
$$
Therefore, $\partial \{u>0\}\cap\{t=t_0\}\cap \mathrm{K}$ is porous with porosity constant $\delta \defeq \frac{\delta^{\ast}}{2}$.
\end{proof}

In contrast to one of the most known properties of heat equation, namely the infinity speed of propagation,  parabolic $p$-dead-core solutions have the property of \textit{finite speed propagation}. Such a property supports the physical soundness of the equation to diffusive models. Moreover, the occurrence of this phenomenon is consequence of the degeneracy of the equation at the level set $u=0$. Such a property is well-known in the literature and can be obtained by several different methods (cf. D\'{i}az \cite[Theorem 10]{Diaz2} for an alternative approach). The next result is similar to the one in \cite[Corollary 4.4]{CW}. Hence, we will only write the modifications for the reader's convenience.

\begin{corollary}[{\bf Finite speed propagation of  $\{u>0\}$}]\label{porousset} There exists a constant $\mathfrak{c}(N, p, q) \geq 1$ such that, for any solution to (\ref{DCSA}), with non-negative and bounded time derivative, and any $Q^{+}_{r}(x_0, t_0) \subset \Omega_T$, the implication:
$$
   u(\cdot, t_0) = 0 \quad \mbox{in} \quad B_r(x_0) \Rightarrow u(\cdot, t_0+\mathfrak{s}^{\theta}) = 0 \quad \mbox{in} \quad B_{\max\{0, r-\mathfrak{c}\mathfrak{s}\}}(x_0)
$$
holds.
\end{corollary}

\begin{proof}
Let us suppose for sake of contradiction that for $0< \mathfrak{s}_1< \frac{r}{\mathfrak{c}}$ there exists a point $x_1 \in B_{ r-\mathfrak{c}\mathfrak{s}_1}(x_0)$ such that $u(x_1, t_0 + \mathfrak{s}_1^{\theta})>0$. The Non-degeneracy property (Theorem \ref{NonDeg}) implies that:
$$
  u(x_2, \varsigma) \geq \mathfrak{C}^{*}_0\mathfrak{s}_1^{\frac{p}{p-1-q}}
$$
for some $(x_2, \varsigma) \in \overline{Q^{-}_{\mathfrak{s}_1}(x_1, t_0+\mathfrak{s}_1^{\theta})}$. Moreover, since $\frac{\partial u}{\partial t}$ is non-negative and bounded we deduce that there exists $0<\tau(n, p, q)<1$ and $(x_2, t_0+ \mathfrak{s}_2^{\theta})$ satisfying:
$$
u(x_2, t_0+ \mathfrak{s}_2^{\theta}) >0, \quad \mbox{with} \quad 0\leq \mathfrak{s}_2 \leq (1-\tau)\mathfrak{s}_1 \quad \mbox{and} \quad |x_2-x_1| \leq \mathfrak{s}_1.
$$
By iterating the previous reasoning we can obtain a point $(x_k, t_0+\mathfrak{s}_k^{\theta})$ so that:
\begin{eqnarray*}
u(x_k, t_0+ \mathfrak{s}_k^{\theta}) >0, \quad \mbox{with} \quad 0\leq \mathfrak{s}_k \leq (1-\tau)^{k-1}\mathfrak{s}_1 \quad \mbox{and} \quad |x_k-x_1| \leq \frac{\mathfrak{s}_1[1-(1-\tau)^{k-1}]}{\tau}.
\end{eqnarray*}
Finally, up to a subsequence $x_k \to x_{\infty}$ as $k \to \infty$, thus we obtain a point $(x_{\infty}, t_0) \in \overline{\{u>0\}}$ fulfilling $|x_{\infty}-x_1|<\frac{\mathfrak{s}_1}{\tau}$. However, this contradicts our assumptions provided $\mathfrak{c} \geq \frac{4}{\tau}$. This contradiction proves the corollary.
\end{proof}


\section{Global analysis results}\label{FurtAppCons}

\textbf{Blow-up analysis.} Throughout this Section we shall study the blow-up analysis over free boundary points (interior touching points). Thus, let $u$ be a solution to  \eqref{DCSA} and $(x_0, t_0) \in \partial \{u > 0\} \cap Q_{\frac{1}{2}}$. Now, consider, for each $\varepsilon>0$, the blow-up family $u_{\varepsilon}: Q_{\frac{1}{2}} \to \R$ given by:
$$
     u_{\varepsilon} (x, t) \defeq \frac{u\left(x_0+ \varepsilon x, t_0+ \varepsilon^{\theta} t\right)}{\varepsilon^{\frac{p}{p-q-1}}}.
$$
We must stress that this sequence is indeed an $\varepsilon$-zoom-in of $u$ re-scaled in a suitable way. Let us analyse the ``limiting profiles''. Note that $u_{\varepsilon}$ fulfils in the weak sense:
$$
    \Delta_p\, u_{\varepsilon} - \frac{\partial u_{\varepsilon}}{\partial t} =  \lambda_0(x_0+ \varepsilon x, t_0+ \varepsilon^{\theta} t).(u_{\varepsilon})_{+}^q \quad \mbox{in} \quad Q_{\frac{1}{2\varepsilon}}.
$$
From Theorem \ref{MThm} we have that:
$$
    u_{\varepsilon}(x, t) \leq C(N, p, q, \lambda_0) \quad \forall \,\,(x, t) \in Q_{\frac{1}{2\varepsilon}}.
$$
Particularly, $u_{\varepsilon}$ is locally bounded in $Q_{\frac{1}{2\varepsilon}}$. From universal H\"{o}lder regularity, see for instance \cite{DB93}, \cite{DiBUV} and \cite{Urbano}, up to a subsequence $u_{\varepsilon} \to u_0$ locally uniformly to an entire function.

From now on,  $u_0$ will  denote a limiting function coming from the previous reasoning. For this reason, we will label it as \textit{Blow-up solution} at $(x_0, t_0) \in \partial \{u>0\}$. The next Theorem \ref{BUprop} establishes a quantitative control profile at infinity for a class of entire solutions to the dead-core problem, namely blow-up solutions (compare with \cite[Section 3.1]{Tei4}).

\begin{theorem}[{\bf Behaviour of Blow-up solutions}]\label{BUprop} Let $u_0$ be a Blow-up solution at $(x_0, t_0) \in \partial \{u>0\}$. Then $u_0(0, 0) = 0$ and
$$
     \Delta_p \, u_{0} - \frac{\partial u_{0}}{\partial t} =  \lambda_0(0, 0).(u_{0})_{+}^q(x, t) \quad \text{in} \quad \R^N\times \R
$$
in the weak sense. Moreover, there exist universal constants $ c_0, C_0 >0$ such that:
\begin{equation}\label{BUest}
    c_0 \leq  \displaystyle \liminf_{|x|+|t|^{\frac{1}{\theta}} \to \infty} \frac{u_0(x, t)}{\left(|x|+|t|^{\frac{1}{\theta}}\right)^{\frac{p}{p-1-q}}} \leq \displaystyle \limsup_{|x|+|t|^{\frac{1}{\theta}} \to \infty} \frac{u_0(x, t)}{\left(|x|+|t|^{\frac{1}{\theta}}\right)^{\frac{p}{p-1-q}}}\leq C_0.
\end{equation}
\end{theorem}

\begin{proof} Since each $u_{\varepsilon}$ fulfils $u_{\varepsilon}(0, 0) = 0$ and:
$$
    \Delta_p u_{\varepsilon} - \frac{\partial u_{\varepsilon}}{\partial t} =  \lambda_0.(u_{\varepsilon})_{+}^{q}
$$
in the weak sense, we can  finish the proof by using stability results (compactness) as in the proof of Lemma \ref{LemmaIter} (see also \cite{DB93}, \cite{DiBUV} and \cite{Urbano}). Finally, the lower and upper control at infinity come from Theorems \ref{MThm} and Theorem \ref{NonDeg}.
\end{proof}

\begin{remark} Note that Blow-up solutions are non-trivial. Moreover, such solutions fulfils:
$$
   c_0.r^\frac{p}{p-q-1} \leq \mathcal{S}_r[u_0]\leq C_0.r^{\frac{p}{p-q-1}},
$$
for values of $r$ large enough.
\end{remark}

\begin{remark}In view of Theorem \ref{BUprop} the non-trivial space-independent blow up solution $u=u(t)$ to:
$$
  \Delta_p u -\frac{\partial u}{\partial t}=\lambda_0u_+^{q}
$$
is given by:
$$
   u(t)=\left[(1-q)\lambda_0(t_0-t)\right]_+^{\frac{1}{1-q}}.
$$
On the other hand, non-trivial time-independent blow up solutions $u=u(x)$ are of the following form:
$$
  u(x) =  \left\lbrace C_{p,q}. (x_i)_{+}^{\frac{p}{p-1-q}}, C_{p,q}. (x_i)_{-}^{\frac{p}{p-1-q}}, C_{p,q}. (|x-x_0|-\mathfrak{R}_0)_{+}^{\frac{p}{p-1-q}}   \right\rbrace,
$$
for any $i=1, \cdots, N$,  where:
$$
   C_{p,q} \defeq \left(\lambda_0. \frac{(p-q-1)^{p}}{p^{p-1}(pq+N(p-1-q))}\right)^{\frac{1}{p-q-1}}.(\pm x)_{+}^{\frac{p}{p-q-1}}.
$$
Notice that the first ones blow-up type solutions are half-space solutions and the last one is a radial solutions with dead core being precisely $B_{\mathfrak{R}_0}(x_0)$.
\end{remark}

In the next paragraph, we shall be devoted to establish a general Liouville-type theorem for any entire solution to \eqref{DCSA}.

\textbf{A Liouville-type Theorem} Liouville-type results have represented an important chapter in the modern mathematical history mainly due to their quantitative and classificatory character. Such results have presented several implications in Analysis, PDEs, Geometry and free boundary problems. In effect, the knowledge of the asymptotic behaviour/global profile for certain entire solutions at infinity is a decisive information in many quantitative researches (cf. \cite{dBGV} for an interesting work on this direction).

In this part, we are concerned  in proving a Liouville-type result for global dead-core solutions. In other words, a global weak solution must grow faster than $\max\left\{|x|, \sqrt[\theta]{|t|}\right\}^{\frac{p}{p-q-1}}$ as $\max\left\{|x|, \sqrt[\theta]{|t|}\right\}\to \infty$, unless it is identically zero. The next Theorem is based on \cite[Theorem 8]{Tei4}, thus we will include some details for completeness.

\begin{theorem}\label{Liouville}
Let $u$ be an entire weak solution to:
$$
    \Delta_p \,u(x, t) - \frac{\partial u}{\partial t}(x, t) = \lambda_0(x, t).u_{+}^q(x, t)
$$
with $u(0, 0)=0\,$ and $\lambda_0$ as before. If $u(x, t)=\mathrm{o}\left(\max\left\{|x|, \sqrt[\theta]{|t|}\right\}^{\frac{p}{p-q-1}}\right)$ as
$\max\left\{|x|, \sqrt[\theta]{|t|}\right\}\to \infty$, then $u\equiv 0$.
\end{theorem}

\begin{proof}
For each positive number $r \gg 1$, let us define the auxiliary scaled function $u_r: Q_1 \to \R_{+}$ given by
$$
     u_r(x, t) \defeq  \frac{u(r\, x,r^{\theta}t)}{r^{\frac{p}{p-q-1}}}.
$$
Thus, it is easy to check that
$$
     \Delta_p \, u_r- \frac{\partial u_r}{\partial t}  = \lambda_0(r\, x,r^{\theta}t) (u_r)_{+}^q \quad \mbox{in} \quad Q_1
$$
in the weak sense and $u_r(0,0)=0$. Now, we affirm that
$$
     \|u_r\|_{L^\infty(Q_1)} = \mbox{o}\,(1) \quad \text{as} \quad r \to \infty.
$$
In effect, for each $r \in \mathbb{R}_{+}$, let $(x_r, t_r) \in \mathbb{R}^N \times \R$ be such that $u_r$ achieves its maximum, i.e.,
$$
    u_r(x_r,t_r)= \sup\limits_{Q_1}u_r(x, t).
$$
Now, we must analyse two possibilities:
\begin{enumerate}
  \item  If $\lim\limits_{r \to \infty} \max\left\{|rx_r|, \sqrt[\theta]{|r^\theta
t_r|}\right\} = \infty,$ we obtain by using the assumption
\begin{align*}
 u_r(x_r,t_r) &= \frac{u(rx_r,r^{\theta}t_r)}{\max\left\{|rx_r|, \sqrt[\theta]{|r^\theta
t_r|}\right\}^{\frac{p}{p-q-1}}}
\max\left\{|x_r|, \sqrt[\theta]{|t_r|}\right\}^{\frac{p}{p-q-1}}\\
&\leq C(N, p, q).o(1)\to 0 \quad \mbox{as}\quad r \to \infty.
\end{align*}
  \item On the other hand, if $\lim\limits_{r \to \infty} \max\left\{|rx_r|, \sqrt[\theta]{|r^\theta
t_r|}\right\} < \infty$ we derive the same conclusion as before for $u_r(x_r, t_r)$, since $u$ is a continuous function (cf. \cite{DB93}, \cite{dBGV},  and \cite{Urbano} for regularity of in-homogeneous degenerate evolution equations).
\end{enumerate}
Therefore, by applying Theorem \ref{MThm0} we obtain
\begin{equation}\label{odeum}
\begin{array}{rcl}
   u_r(x, t) & \leq & \mbox{o}\,(1).\mathfrak{C}(N, p, q, \lambda_0) \left(|x|+|t|^{\frac{1}{\theta}}\right)^{\frac{p}{p-q-1}} \\
   & \leq & \mbox{o}\,(1).\mathfrak{C}(N, p, q, \lambda_0)\max\left\{|x|, \sqrt[\theta]{|t|}\right\}^{\frac{p}{p-q-1}} \quad
\mbox{in} \quad Q_{\frac{1}{2}} \quad \text{for} \quad r\gg 1,
\end{array}
\end{equation}
where we have used the equivalence of norms in $\R^{N+1}$. Now, suppose reach a contradiction that there exists a $(x_0, t_0) \in (\mathbb{R}^N \times \R_{+}) \setminus \{(0, 0)\}$ such that $u(x_0, t_0)>0$. Observe that \eqref{odeum} means that given $\xi>0$ there exists an $r_0 = r_0(\xi, p, q)>0$ such that
$$
            \sup\limits_{Q_{\frac{1}{2}}} \dfrac{u_r(x,t)}{\max\left\{|x|, \sqrt[\theta]{|t|}\right\}^{\frac{p}{p-q-1}}} \le \xi,
$$
provided $r> r_0$. For such a $\xi$ fixed, we now estimate, for $r \gg 2\max\left\{|x_0|, \sqrt[\theta]{|t_0|}, r_0\right\}$:
    \begin{eqnarray*}
            \dfrac{u(x_0,t_0)}{\max\left\{|x_0|, \sqrt[\theta]{|t_0|}\right\}^{\frac{p}{p-q-1}}} \le \sup\limits_{Q_{\frac{r}{2}}} \dfrac{u(x,t)}{\max\left\{|x|, \sqrt[\theta]{|t|}\right\}^{\frac{p}{p-q-1}}}
            \le  \sup\limits_{Q_{\frac{1}{2}}} \dfrac{u_r(x, t)}{\max\left\{|x|, \sqrt[\theta]{|t|}\right\}^{\frac{p}{p-q-1}}}
         \le  \xi.
    \end{eqnarray*}
The proof finishes  by selecting $0<\xi < \dfrac{u(x_0,t_0)}{\max\left\{|x_0|, \sqrt[\theta]{|t_0|}\right\}^{\frac{p}{p-q-1}}}$, which is possible due to our  hypothesis. Such a contradiction completes the proof of Theorem.
\end{proof}

\begin{remark} Consider:
$$
  u(x, t)  = \left\{
\begin{array}{l}
  \left[\lambda_0(1-q)(-t)\right]_+^{\frac{1}{1-q}}\\
  C_{p,q}. (x_i)_{\pm}^{\frac{p}{p-1-q}} \\
   C_{p,q}. (|x-x_0|-\mathfrak{R}_0)_{+}^{\frac{p}{p-1-q}}
\end{array}
\right.
$$
for $i=1, \cdots, N$, where
$$
   C_{p,q} \defeq \left(\lambda_0. \frac{(p-q-1)^{p}}{p^{p-1}(pq+N(p-1-q))}\right)^{\frac{1}{p-q-1}}.
$$
Then, $u$ is a weak solutions to
$$
    \Delta_p \,u(x, t) - \frac{\partial u}{\partial t}(x, t) = \lambda_0.u_{+}^q(x, t).
$$
Then, the assumption of Theorem \ref{Liouville} is sharp in the sense that there are half-space/ radial solutions such that
$$
  \frac{u(x, t)}{\left(\max\left\{|x|, \sqrt[\theta]{|t|}\right\}^{\frac{p}{p-q-1}}\right)}> 0 \quad \text{as}
 \quad \max\left\{|x|, \sqrt[\theta]{|t|}\right\}\to \infty.
$$
For an analysis about radial solutions of fully non-linear elliptic and quasi-linear dead-core problems we recommend \cite[Section 6]{LRS} and \cite[Section 5.2]{SS}.
\end{remark}

\section*{Acknowledgements}

This work was partially supported by ANPCyT under grant PICT $2012-0153$, by Consejo Nacional de Investigaciones Cient\'{i}ficas y T\'{e}cnicas (CONICET-Argentina PIP $5478/1438$), by the CNPq (Brazilian Government Program \textit{Ci\^{e}ncia sem Fronteiras}) and by Universidad de Buenos Aires under grant UBACYT $20020100100400$. The authors would like to thank Prof. Jes\'{u}s Ildefonso D\'{i}az for pointing out many valuable references. The authors would also like to thank Prof. Noem\'{i} Wolanski for several insightful comments and discussions. The authors are also grateful to the anonymous referee(s) for pointing out a number of improvements that benefited a lot the final outcome of the article. JV da Silva thanks IMASL (CONICET) from Universidad Nacional de San Luis for its warm hospitality and for fostering a pleasant scientific atmosphere during his visit where part of this paper was written. JV da Silva also thanks (CONICET - Postdoctoral Fellowship) and \textit{Research Group of PDEs}/Math. Dept./FCEyN from Universidad de Buenos Aires by the excellent working environment during his Postdoctoral program. P. Ochoa and A. Silva were partially supported by CONICET-Argentina.

\end{document}